\documentclass[12pt]{article}

\usepackage{caption}
\usepackage{xparse}
\usepackage{amsmath}
\usepackage{amsfonts,amsthm,amssymb,bbding}
\usepackage{float}
\usepackage{graphicx}
\usepackage{enumerate}
\usepackage{tikz}
\usepackage{caption}
\usepackage{rotating}
\allowdisplaybreaks[4]
\usepackage{hyperref}
\hypersetup{
	colorlinks=true,
	anchorcolor=yellow,
	linkcolor=cyan,
	filecolor=blue,
	urlcolor=red,
	citecolor=green,
}
\usepackage[numbers,sort&compress,comma,square]{natbib}
\usepackage{latexsym,bm}
\usepackage{mathtools}

\evensidemargin=\oddsidemargin\topmargin=-1.5cm

\theoremstyle{definition}

\usepackage[margin=1in]{geometry}

\theoremstyle{definition}
\newtheorem{defn}{Definition}[section]
\theoremstyle{definition}

\theoremstyle{plain}

\newtheorem{thm}[defn]{Theorem}
\newtheorem{lem}[defn]{Lemma}
\newtheorem{cor}[defn]{Corollary}
\newtheorem{prop}[defn]{Proposition}
\theoremstyle{remark}

\newtheoremstyle{case}{}{}{}{}{}{:}{ }{}
\theoremstyle{case}
\newtheorem{case}{\textbf{Case}}

\newcommand{\etal}{\mbox{\textit{et al}.\ }}

\newcommand{\ii}{\mathtt{i}}

\newcommand{\R}{\mathbb{R}}
\newcommand{\Z}{\mathbb{Z}}
\newcommand{\Q}{\mathbb{Q}}

\newcommand{~}{\sim}

\DeclareMathOperator{\Sp}{Sp}

\newcommand{\allone}[1]{{j}_{#1}}

\newcommand{\twovector}[2]{\left[\begin{array}{c} #1 \\ #2 \end{array}\right]}

\NewDocumentCommand\eigenA{mg}{%
	\ensuremath{E_{#1}{\IfNoValueTF{#2}{}{(#2)}}}%
}

\ifthenelse{\isundefined{\usedirac}}
{ 
	\newcommand{\ket}[1]{\mathbf{#1}}
	\newcommand{\bra}[1]{\mathbf{#1}^{\dagger}}
	\newcommand{\tbra}[1]{\mathbf{#1}^{T}}
	\newcommand{\eket}[1]{{\mathbf{e}}_{#1}}        
	\newcommand{\ebra}[1]{\eket{#1}^{T}}            

}
{ 
	\newcommand{\bra}[1]{\left\langle #1 \right|}
	\newcommand{\ket}[1]{\left| #1 \right\rangle}
	\newcommand{\ebra}[1]{\bra{#1}}
	\newcommand{\eket}[1]{\ket{#1}}

}
\begin{document}
\title{\bf Perfect state transfer in neighborhood coronas}
\author{Xing-Kun Song\footnote{Email address: \href{mailto:xksong@126.com}{xksong@126.com} (X.-K. Song)} \\[2mm]
	\small School of Mathematics, East China University of Science and Technology, \\
	\small  Shanghai 200237, P.R. China\\}

\date{}
\maketitle
{\flushleft\large\bf Abstract}
The neighborhood corona $G \star H$ is the graph obtained by taking one copy $G$ and $|G|$ copies of $H$, and joining each vertex of the $j$th copy of $H$ to all neighbors of $v_{j}$ in $G$. In this paper, we study the state transfer of neighborhood coronas related to the adjacency matrix. Concretely, we provide some necessary conditions under which the neighborhood corona $G \star H$ admits perfect state transfer, and obtain a new family of graphs with pretty good state transfer based on neighborhood coronas.

\begin{flushleft}
\end{flushleft}

\section{Introduction}

A continuous-time quantum walk is a quantum walk on a given graph that is dictated by a transition matrix which relies on the Hamiltonian of the quantum system and the adjacency matrix. The concept was  introduced by Farhi and Gutmann \cite{FG98} for
quantum computation in 1998. Let $G$ be a graph with adjacency matrix $A_{G}$. The transition matrix of $G$ is defined by
\begin{equation*}
U(t) = \exp(-\ii t A_{G}) = \sum_{k \geqslant 0} \frac{(-\ii)^{k} A_{G}^{k} t^{k}}{k!}, t \in \R^{+}, \ii = \sqrt{-1}.
\end{equation*}
Note that $U(t)$ is both symmetric and unitary.
Given two distinct vertices $u$ and $v$ in $G$, we say that $G$ admits {\em perfect state transfer} from vertices $u$ to $v$ if there exists a time $t \in \R^{+}$ and a complex unimodular scalar $\gamma$ such that
\begin{equation} \label{equ::2}
U(t)\eket{u} = \gamma \eket{v}.
\end{equation}
Here $\gamma$ is called the {\em phase} of perfect state transfer. The concept of perfect state transfer was introduced by Bose \cite{B03} in 2003.
In particular,  if $u = v$ in Eq. \eqref{equ::2},
we say that $G$ is {\em periodic at vertex $u$}. Furthermore, if $U(t)$ is a scalar multiple of the identity matrix, then $G$ is {\em periodic}. In the past two decades, the graphs with perfect state transfer have aroused a lot of interest, such graphs are rare. For this reason, Godsil \cite{g11-dm} proposed the concept of pretty good state transfer. We say that a graph $G$ has {\em pretty good state transfer} from $u$ to $v$ if
for any $\varepsilon > 0$, there exists a time $t \in \R^{+}$ such that
\begin{equation*}
|U(t)_{u,v}| > 1 - \varepsilon.
\end{equation*}

Up to now, some interesting results have been obtained for perfect state transfer or pretty good state transfer on graphs. Godsil \cite{g12} proved that for any integer $k$ there are only finitely many graphs of maximum degree $k$ admitting perfect state transfer.
Christandl et al. \cite{CDDEKL05,CDEL04} observed that the path $P_{n}$ on $n$ vertices
admits antipodal perfect state transfer if and only if $n=2,3$.
Fan and Godsil  \cite{fg13} showed that a double star of order $4m+1$ admits  pretty good state transfer between the two central vertices if and only if $4m+1$ is not a perfect square. van Bommel \cite{vB19} obtained a complete characterization of pretty good state transfer on paths.
Pal et al. \cite{PB173,Pa18,Pa19} investigated the existence of pretty  good state transfer on circulant graphs. Ackelsberg et al. \cite{ABCMT17} studied perfect state transfer and pretty good state transfer based on coronas. In recent years, perfect state transfer or pretty good state transfer based on some other graph products such as edge corona,  edge complemented corona and NEPS are also investigated by researchers \cite{PB16,PB17,PB172,ABCMT16,TYC19,ZLZ20,LLZ20,LLZZ21,LZ21,WL21}. For more results on  this topic, we refer the reader to \cite{cggv15,ANOPRT10,Ba11,Co16,CG11,ZBS20} and the surveys \cite{g11,KT11,CG21}.

The concept of the neighborhood corona was introduced by Gopalapillai \cite{Go11} in 2011.  Given two graphs $G$ and $H$ on $n$ and $m$ vertices, respectively, the \textit{neighborhood corona of} $G$ and $H$, denoted by $G \star H$, is the graph obtained by taking one copy of $G$ and $n$ copies of $H$, all vertex-disjoint, and joining every neighbor of the $i$th vertex of $G$ to every vertex in the $i$th copy of $H$ by a new edge.  In this paper, we consider a new family of graphs with perfect state transfer and pretty good state transfer based on neighborhood coronas. We find that there is no perfect state transfer on neighborhood coronas except some special situations, and we construct a family of graphs with pretty good state transfer based on neighborhood coronas.

The rest of the paper is organized as the follows. In Section \ref{sec::2}, some basic concepts and useful results are given. In Section \ref{sec::3}, we compute eigenvalues, eigenprojectors, and spectral decompositions of neighborhood coronas, respectively. In Section \ref{sec::4} and Section \ref{sec::5}, we investigate the existence of perfect state transfer and pretty good state transfer in neighborhood coronas.
\section{Preliminaries}\label{sec::2}

Let $G$ be a graph on $n$ vertices. The \textit{adjacency matrix} of $G$, denoted by $A_{G}$, is the $n\times n$ matrix such that for a pair of vertices $u$ and $v$, the $(u,v)$ entry of $A_{G}$ is $1$ if $u,v$ are adjacent, and $0$ otherwise.
The \textit{spectrum} of $G$, denoted by $\Sp(G)$, is the multiset of eigenvalues of $A_{G}$, and the \textit{spectral radius} of $G$, denoted by $\rho(G)$, is the largest eigenvalue of  $G$.

Let $\lambda_1,\ldots,\lambda_d$ denote all distinct eigenvalues of  $A_{G}$. By the spectral decomposition, we have
\begin{equation*}
	A_{G}
= \sum_{r=1}^{d} \lambda_{r} E_{\lambda_{r}},
\end{equation*}
where $E_{\lambda_{r}}$ is the eigenprojector corresponding to $\lambda_{r}$. Note that $\sum_{r=1}^{d} E_{\lambda_{r}} =I$, $E_{\lambda_{r}}^{2}=E_{\lambda_{r}}$, and
$E_{\lambda_{r}} E_{\lambda_{s}} = \mathbf{0}$ for $r \neq s$, where $I$ and $\mathbf{0}$ are  the identity matrix and the zero matrix. Then we see that
\begin{equation*}
U_{A_{G}}(t) =
\sum_{k \geqslant 0} \frac{(-\ii)^{k} A_{G}^{k} t^{k}}{k!} =
\sum_{k \geqslant 0} \frac{(-\ii)^{k} (\sum_{r=1}^{d} \lambda_{r} E_{\lambda_{r}})^{k} t^{k}}{k!} =
\sum_{r=1}^{d} \exp(-\ii t \lambda_{r}) E_{\lambda_{r}}.
\end{equation*}

The {\em eigenvalue support} of a vertex $u$ in $G$, denoted by $\Phi_{u}$,
is the set of all eigenvalues $\lambda$ of $A_{G}$ such that $\eigenA{\lambda}{G}\eket{u} \neq 0$,
where $\eket{u}$ is the characteristic vector corresponding to $u$.
Two vertices $u$ and $v$ of $G$ are called {\em strongly cospectral} if
$\eigenA{\lambda}{G}\eket{u} = \pm\eigenA{\lambda}{G}\eket{v}$
for every eigenvalue $\lambda$ of $A_{G}$.

First of all, we state some useful results about perfect state transfer and periodicity of graphs.

Given a rational $m$ and a prime $p$, we can write $m = p^\alpha\frac{r}{s}$, where $r$ and
$s$ are integers not divisible by $p$. Then the $p$-adic norm of $m$ is defined by
\[
|m|_{p} = p^{-\alpha}.
\]

Thus, if $m$ is integer, the larger the power of $p$ dividing $m$ is, the smaller
its $p$-adic norm is.

\begin{thm}[Coutinho and Godsil \cite{CG21}] \label{thm::2.1}
Let $G$ be a graph and $u$ and $v$ be two vertices of $G$, and assume the eigenvalue support of vertex $u$ consists of eigenvalues $\lambda_{0}> \cdots> \lambda_{k}$. Then $G$ admits perfect state transfer between $u$ and $v$ if and only if:
\begin{enumerate}[(i)]
\item the two vertices $u$ and $v$ are strongly cospectral;

\item the eigenvalues in $\Phi_{u}$ are either integers or quadratic integers, and, moreover, there are integers $a, \Delta, b_{0}, \ldots, b_{k}$, with $\Delta$ positive and square-free, so that
\begin{equation*}
	\lambda_{r}=\frac{1}{2}(a+b_{r} \sqrt{\Delta});
\end{equation*}
\item  there is a non-negative integer $\alpha$ so that
\begin{itemize}
	\item $\left(E_{r}\right)_{a, b}>0$ if and only if $|\left(\lambda_{0}-\lambda_{r}\right) / \sqrt{\Delta}|_{2}<2^{-\alpha}$,
	\item $\left(E_{r}\right)_{a, b}<0$ if and only if $|\left(\lambda_{0}-\lambda_{r}\right) / \sqrt{\Delta}|_{2}=2^{-\alpha}$.
\end{itemize}

\end{enumerate}
If the above conditions hold, let
\begin{equation*}
g=\operatorname{gcd}\left( \left\{\frac{\lambda_{0}-\lambda_{r}}{\sqrt{\Delta}}:r=0, \ldots, k\right\}\right),
\end{equation*}
then the minimum time we have perfect state transfer between $a$ and $b$ is $\tau=\pi / g \sqrt{\Delta}$, and any other time it occurs is an odd multiple of $\tau$.
\end{thm}

\begin{lem}[Godsil \cite{g11}] \label{lem::2.2}
If $G$ has perfect state transfer between vertices $u$ and $v$ at time $t$,
then $G$ is periodic at $u$ at time $2t$.
\end{lem}

\begin{thm}[Godsil \cite{g12}] \label{lem::2.3}
A graph $G$ is periodic at vertex $u$ if and only if either:
\begin{enumerate}[i)]
\item all eigenvalues in $\Phi_{u}$ are integers; or
\item there is a square-free integer $\Delta$ and an integer $a$ so that
	each eigenvalue $\lambda$ in $\Phi_{u}$ is of the form
	$\lambda = \frac{1}{2}(a + b_{\lambda}\sqrt{\Delta})$,
	for some integer $b_{\lambda}$.
\end{enumerate}
\end{thm}

To consider pretty good state transfer of graphs, we need the following form of Kronecker's approximation theorem.

\begin{thm}[Kronecker's Theorem \cite{hw00}] \label{thm::2.4}
Let $1,\,\lambda_{1},\,\ldots,\,\lambda_{m}$ be linearly independent over rational numbers $\Q$.
Let $\alpha_{1},\, \ldots,\, \alpha_{m}$ be arbitrary real numbers,
and let $N,\,\varepsilon$ be positive real numbers.
Then there are integers $\ell > N$ and $q_1,\,\ldots, \,q_m$
so that
\begin{equation}\label{eq::8}
|\ell\lambda_{k} - q_{k} - \alpha_{k}| < \varepsilon,
\end{equation}
for each $k=1,\,\ldots,\,m$.
\end{thm}

For simplicity, inequalities of the form $|\alpha - \beta| < \varepsilon$, for arbitrarily small $\varepsilon$,
can be written instead as $\alpha \approx \beta$ and omit the explicit dependence on
$\varepsilon$. For example, inequalities \eqref{eq::8} can be represented as
$\ell\lambda_k - q_k \approx \alpha_k$.

In order to use Kronecker's Theorem, we need the following lemma which gives a set of numbers which are linearly independent over rational numbers $\Q$.

\begin{thm}[\cite{R74}]\label{thm::2.5}
Let $p_{1}, p_{2}, \ldots, p_{n}$ be distinct positive primes. The set $\{\sqrt[n]{p_{1}^{m(1)} \cdots p_{k}^{m(k)}}: 0 \leqslant m(i)<n, 1 \leqslant i\leqslant k\}$ is linearly independent over the set of rational numbers $\Q$.
\end{thm}

Taking $n=2$,  Theorem \ref{thm::2.5} implies the following result immediately.

\begin{cor}\label{cor::2.6}
The set $\{\sqrt{\Delta}$ , $\Delta$ is a square-free integer $\}$ is linearly independent over the set of rational numbers $\Q$.
\end{cor}

The following symbols are used throughout the rest of the paper:
\begin{itemize}
	\item $\ket{\allone}_{n}$:  The all-one vectors of dimension $n$;
	\item $\ket{0}_{n}$: The all-zero vectors of dimension $n$;
	\item $J_{m,n}$: The $m \times n$ all-one matrix;
	\item  $J_{n}$: The $n \times n$ all-one matrix;
	\item  $\eket{n}^{i}$: The $n\times 1$ column vector with $i$th entry is one, and all other entries is zero;
	\item $I_{n}$: The identity matrix of dimension $n$.
\end{itemize}

\section{Neighborhood Coronas of Graphs}\label{sec::3}

Let $G$ be a graph with vertex set $V(G)=\{v_1,\ldots, v_n\}$, and $H$ be a graph
with vertex set $V(H)=\{w_1,\ldots, w_m\}$.
The neighborhood corona $G \star H$ has the vertex set
\begin{equation*}
	V(G \star H) = V(G) \times \left( \{0\} \cup V(H) \right),
\end{equation*}
and the adjacency relation
\begin{equation*}
	((v, w), (v', w')) \in E(G \star H)
	\iff
	 \begin{cases}
	 	\text{$w = w' = 0$ and $(v,v') \in E(G)$} & \text{or} \\
	 	\text{$v = v'$ and $(w,w') \in E(H)$}     & \text{or} \\
	 	\text{$w' = 0$ and $v' \in N_{G}(v)$.}       &
	 \end{cases}
\end{equation*}

\begin{figure}[htbp]
	\centering
	\begin{tikzpicture}[scale=0.8]
		\foreach \x in {-2,0,2} {
			\draw (0,0) -- (\x,0);
			\fill (\x,0) circle[radius = 0.1];
			\foreach \y in {0,...,3}{
				\draw (1/2+\x, 2) -- (-1/2+\x, 2);
				\draw (0,0) -- +(-2-1/2 + \y/3, 2);
				\draw (0,0) -- +(2-1/2 + \y/3, 2);
				\draw (2,0) -- +(-2-1/2 + \y/3, 2);
				\draw (-2,0) -- +(2-1/2 + \y/3, 2);
				\filldraw (\x-1/2 + \y/3, 2) circle[radius=0.1];
			}	
		}
	\end{tikzpicture}
	\caption{The neighborhood corona $P_3 \star P_4$.\label{Fig1}}
\end{figure}
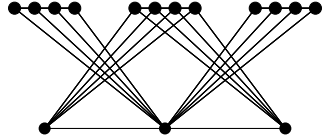

For example, the neighborhood corona $P_3 \star P_4$ is shown in Figure \ref{Fig1}. According to the definition, the adjacency matrix of $G \star H$  is given by
\begin{equation*}
	A(G \star H) =
		\left[\begin{array}{cc}
			A(G)                                   & \tbra{j}_m \otimes A(G) \\
			\ket{j}_m \otimes A(G)^T & A(H) \otimes I_{n}
		\end{array}\right].
\end{equation*}


In \cite[Theorem 2.1]{Go11}, Gopalapillai showed that the spectrum of $G \star H$ is determined by the spectra of $G$ and $H$ under the condition that  $H$ is a $k$-regular graph.

\begin{thm}[Gopalapillai \cite{Go11}]\label{thm::3.1}
Let $G$ be a graph with eigenvalues $\lambda_1 \geqslant \lambda_2 \geqslant \cdots \geqslant \lambda_n$ and $H$ be a $k$-regular graph with eigenvalues $k = \mu_1 \geqslant \mu _2 \geqslant \cdots \geqslant \mu_m$. Let $\ket{x}_{i}$ be a eigenvector of $A(G)$ with an eigenvalue $\lambda_{i}$, $i = 1,2, \ldots , n$, and let $\ket{y}_{j}$ be a eigenvector of $A(H)$ with an eigenvalue $\mu_{j}$, $j = 2,3, \ldots , m$.
Then the eigenvalue and the corresponding eigenvector of $G \star H$ consists of:
\begin{enumerate}[(i)]	
	\item $\lambda_{i}^\pm =\frac{1}{2}\left(\lambda_i + k \pm \sqrt{(\lambda_i - k)^2 + 4m \lambda_i^2}\right)$ for each $i = 1,2, \ldots , n$, and the corresponding eigenvector is
\begin{equation*}
  \twovector{\ket{0}_n}{\ket{\allone}_{m} \otimes \ket{x}_{i}}_{\lambda_{i}^{+}=k}, \twovector{\ket{x}_{i}}{\ket{0}_{m} \otimes \ket{x}_{i}}_{\lambda_{i}^{-}=0},
\end{equation*}
 for $\lambda_{i} = 0$, and
\begin{equation*}
  \twovector{\frac{\lambda_{i}^\pm-k}{\lambda_{i}} \ket{x}_{i}}{\ket{\allone}_{m} \otimes \ket{x}_{i}},
\end{equation*}
 for $\lambda_{i} \neq 0$.

\item $\mu_j$, with multiplicity n, for each $j = 2,3, \ldots , m$, and the corresponding eigenvector is
\begin{equation*}
  \twovector{\ket{0}_{n}}{\ket{y}_{j} \otimes \eket{n}^{i}}.
\end{equation*}

\end{enumerate}
\end{thm}
We now construct the eigenprojectors of $G \star H$ by the eigenprojectors of $G$ and $H$, which are essential in the subsequent analyses.

\begin{prop} \label{prop::3.2}
Let $G$ be a connected graph on $n$ vertices and $H$ be a connected $k$-regular  graph on $m$ vertices. Suppose that $\Sp(G)=\{\lambda_1, \lambda_2, \cdots, \lambda_p\}$ and $\Sp(H)=\{\mu_1, \mu_2, \cdots, \mu_q\}$ are the sets of  distinct eigenvalues of $G$ and $H$, respectively, where $\lambda_1>\lambda_2>\cdots> \lambda_p$ and $k=\mu_1>\mu_2>\cdots> \mu_q$.
\begin{enumerate}[(i)]
\item \label{item::prop::3.2::1}
For each eigenvalue $\mu_{j}$ of $H$, $j = 2,3, \ldots , q$, the  eigenprojector of $G \star H$ with respect to $\mu_j$ is
\begin{equation*}
	\eigenA{\mu_{j}}(G \star H)
	:=
	\left[\begin{array}{cc}
		0           & \ket{0}_{m}^T                                   \\
		\ket{0}_{m} & \eigenA{\mu_{j}}{H}
	\end{array}\right] \otimes I_{n}.
\end{equation*}
\item \label{item::prop::3.2::2}
For each eigenvalue $\lambda_{i}$ of $G$,  $i = 1,2, \ldots , p$, let
\begin{equation*}
	\lambda^\pm_{i} := \frac{\lambda_{i} + k \pm \sqrt{(\lambda_{i} - k)^{2} + 4m \lambda_{i}^2}}{2}.
\end{equation*}
The  eigenprojector of $G \star H$ with respect to $\lambda^\pm_{i}$ is
\begin{equation*}
	\eigenA{k}(G \star H)
	:= \frac{1}{m}
	\left[\begin{array}{cc}
		0           & \ket{0}_{m}^T \\
		\ket{0}_{m} & J_{m}
	\end{array}\right]\otimes \eigenA{0}{G},
\end{equation*}
\begin{equation*}
\eigenA{0}(G \star H)
	:= \left[\begin{array}{cc}
		1           & \ket{0}_{m}^T \\
		\ket{0}_{m} & \ket{0}_{m\times m}
	\end{array}\right]\otimes \eigenA{0}{G},
\end{equation*}
if  $\lambda_{i} = 0$,
and
\begin{equation*}
	\eigenA{\lambda^\pm_{i}}(G \star H)
	:= \frac{1}{(\lambda^\pm_{i} - k)^2 +m\lambda_{i}^2}
	\left[\begin{array}{cc}
		(\lambda^\pm_{i} - k)^2                   & \lambda_{i}(\lambda^\pm_{i} - k)\tbra{j}_m \\
		\lambda_{i}(\lambda^\pm_{i} - k)\ket{j}_m & \lambda_{i}^2 J_{m}
	\end{array}\right] \otimes \eigenA{\lambda_{i}}(G),
\end{equation*}
if  $\lambda_{i} \neq 0$.

\item \label{item::prop::3.2::3}
The spectral decomposition of  $G \star H$ is given by
\begin{equation*}
	A(G \star H) =
	\sum_{\lambda_{i} \in \Sp(G)} \sum_\pm \lambda^\pm_{i} \eigenA{\lambda^\pm_{i}}(G \star H)
	 + \sum_{\mu_{j} \in \Sp(H)\setminus\{k\}} \mu_{j} \eigenA{\mu_{j}}(G \star H).
\end{equation*}
\end{enumerate}

\begin{proof}
For (\ref{item::prop::3.2::1}), let  $B_{\mu_{j}}$  be an orthonormal basis of the $\mu_j$-eigenspace of $H$.  Then  the eigenprojector of $\mu_{j}$ in $H$ is
\begin{equation*}
	\eigenA{\mu_{j}}{H} =
		\sum_{\ket{y}_{j} \in B_{\mu_{j}}} \ket{y}_{j}\ket{y}_{j}^T.
\end{equation*}	
By Theorem \ref{thm::3.1}, an orthonormal basis of the $\mu_{j}$-eigenspace of $G \star H$
is given by
\begin{equation*}
\left\{ \ket{Y}_{j}^{i}=\twovector{\ket{0}_{n}}{\ket{y}_{j} \otimes \eket{n}^{i}}: \ket{y}_{j} \in B_{\mu_{j}}, i = 1,2, \ldots , n\right\}.
\end{equation*}
Thus the eigenprojector  of $G\star H$  with respect to $\mu_{j}$ is
\begin{equation*}
\begin{aligned}
\eigenA{\mu_{j}}(G \star H) = \sum_{\ket{y}_{j} \in B_{\mu_{j}}} \sum_{i=1}^{n} \ket{Y}_{j}^{i} {\ket{Y}_{j}^{i}}^{T} =
	\left[\begin{array}{cc}
	0           & \ket{0}_{m}^T                                   \\
	\ket{0}_{m} & \eigenA{\mu_{j}}{H}
\end{array}\right] \otimes I_{n}.
\end{aligned}
\end{equation*}

For (\ref{item::prop::3.2::2}), let $B_{\lambda_{i}}$ be an orthonormal basis of the  $\lambda_{i}$-eigenspace of $G$. Then the eigenprojector of $\lambda_{i}$ in $G$ is
\begin{equation*}
	\eigenA{\lambda_{i}}{G} = \sum_{\ket{x}_{i} \in B_{\lambda_{i}}} \ket{x}_{i}\ket{x}_{i}^T.
\end{equation*}
Recall that $\lambda_{i}^\pm = \frac{1}{2}\left(\lambda_{i} + k \pm \sqrt{(\lambda_{i} - k)^2 + 4m\lambda_{i}^2}\right)$. If $\lambda_{i} = 0$, then $\lambda_{i}^+ = k$ and $\lambda_{i}^- = 0$.  By  Theorem \ref{thm::3.1}, an orthonormal basis of the $k$-eigenspace of $G \star H$
is given by
\[\left\{\ket{X}_{k} = \frac{1}{\sqrt{m}}\twovector{\ket{0}_n}{\ket{\allone}_{m} \otimes \ket{x}_{i}}:\ket{x}_{i}\in B_{0}\right\};
\]
an orthonormal basis of the $0$-eigenspace of $G \star H$
is given by
\[\left\{\ket{X}_{0} =\twovector{\ket{x}_{i}}{\ket{0}_{m} \otimes \ket{x}_{i}}:\ket{x}_{i}\in B_{0}\right\}.
\]
Therefore, if $\lambda_{i} = 0$, then the eigenprojector  of $G\star H$  with respect to $k$ is
\begin{equation*}
\begin{aligned}
\eigenA{k}(G \star H) &= \sum_{\ket{x}_{i} \in B_{0}} \ket{X}_{k}{\ket{X}_{k}}^T =
 \frac{1}{m}
\left[\begin{array}{cc}
	0           & \ket{0}_{m}^T \\
	\ket{0}_{m} & J_{m}
\end{array}\right]\otimes \eigenA{0}{G};
\end{aligned}
\end{equation*}

the eigenprojector  of $G\star H$  with respect to $0$ is
\begin{equation*}
\begin{aligned}
\eigenA{0}(G \star H) &= \sum_{\ket{x}_{i} \in B_{0}} \ket{X}_{0}{\ket{X}_{0}}^T =
 \left[\begin{array}{cc}
		1           & \ket{0}_{m}^T \\
		\ket{0}_{m} & \ket{0}_{m\times m}
	\end{array}\right]\otimes \eigenA{0}{G}.
\end{aligned}
\end{equation*}

If $\lambda_{i} \neq 0$, then $\lambda_{i}^+ \neq k$, $\lambda_{i}^- \neq 0$.  By  Theorem \ref{thm::3.1}, an orthonormal basis of the $\lambda_{i}^\pm$-eigenspace of $G \star H$
is given by
$$\left\{\ket{X}_{i}^{\pm} =\frac{1}{\sqrt{(\frac{\lambda^\pm_{i}-k}{\lambda_{i}})^2+m}}\twovector{\frac{\lambda^\pm_{i}-k}{\lambda_{i}} \ket{x}_{i}}{\ket{\allone}_{m} \otimes \ket{x}_{i}}:\ket{x}_{i}\in B_{\lambda_{i}}\right\}.$$
Therefore, if $\lambda_{i} \neq 0$, then the eigenprojector  of $G\star H$  with respect to $\lambda^\pm_{i}$ is
\begin{equation*}
\begin{aligned}
\eigenA{\lambda^\pm_{i}}(G \star H) &= \sum_{\ket{x}_{i} \in B_{\lambda_{i}}} \ket{X}_{i}^{\pm}{\ket{X}_{i}^{\pm}}^T =
\frac{1}{(\lambda^\pm_{i} - k)^2 +m\lambda_{i}^2}
\left[\begin{array}{cc}
	(\lambda^\pm_{i} - k)^2                   & \lambda_{i}(\lambda^\pm_{i} - k)\tbra{j}_m \\
	\lambda_{i}(\lambda^\pm_{i} - k)\ket{j}_m & \lambda_{i}^2 J_{m}
\end{array}\right] \otimes \eigenA{\lambda_{i}}(G).
\end{aligned}
\end{equation*}
This proves (\ref{item::prop::3.2::2}). Furthermore, from (\ref{item::prop::3.2::1}), (\ref{item::prop::3.2::2}) and Theorem \ref{thm::3.1}, we obtain (\ref{item::prop::3.2::3}) immediately.
\end{proof}
\end{prop}
In what follows, we apply Proposition \ref{prop::3.2} to analyze the transition matrix
$\exp(-\ii tA(G \star H))$, and thereby discuss the state transfer in neighborhood coronas.

\begin{prop} \label{prop::3.3}
Let $G$ be a connected graph on $n$ vertices and $H$ be a connected $k$-regular graph on $m$ vertices.
\begin{enumerate}[(i)]
  \item \label{item::prop::3.3::1}
  For vertices $v$ and $v'$ of $G$, we have
\begin{equation*}
	\ebra{(v,0)} e^{-\ii tA(G \star H)} \eket{(v',0)}
	= \sum_{\lambda \in \Sp(G)} e^{-\ii t(\lambda + k)/2}
		\ebra{v}\eigenA{\lambda}{G}\eket{v'}
		\left( \cos\left(t\Lambda_\lambda/2\right)
			- \frac{\lambda - k}{\Lambda_\lambda} \ii \sin\left(t\Lambda_\lambda/2 \right) \right),
\end{equation*}
where $\Lambda_\lambda = \sqrt{(\lambda - k)^2 + 4m\lambda^2}$, and $\Sp(G)$ is the set of all distinct eigenvalues of $G$.
  \item \label{item::prop::3.3::2}
  For vertices $v$ and $v'$ of $G$, and vertex $w$ of $H$, we have
\begin{equation*}
	\ebra{(v',0)} e^{-\ii tA(G \star H)} \eket{(v,w)}
	= \sum_{\lambda \in \Sp(G)} e^{-\ii t(\lambda + k)/2}
		\ebra{v}\eigenA{\lambda}{G}\eket{v'}
		 \frac{-2\lambda}{\Lambda_\lambda} \ii \sin\left(t\Lambda_\lambda/2 \right),
\end{equation*}
where $\Lambda_\lambda = \sqrt{(\lambda - k)^2 + 4m\lambda^2}$, and $\Sp(G)$ is the set of all distinct eigenvalues of $G$.
\end{enumerate}

\end{prop}
\begin{proof}
For (\ref{item::prop::3.3::1}) Notice that $\lambda^\pm = \frac{1}{2}(\lambda+k \pm \Lambda_{\lambda})$ for each $\lambda\in \Sp(G)$.
By Proposition \ref{prop::3.2}, we obtain
\begin{equation}
\begin{aligned}\label{eq::30}
	\ebra{(v,w)} e^{-\ii tA(G \star H)} \eket{(v',w')} & =
        \sum_{\lambda \in \Sp(G)}
        e^{-\ii t\lambda^\pm}
		\left(\ebra{(v,0)}\eigenA{\lambda^\pm}{G \star H}\eket{(v',0)}\right) \\
	                                                   & \ + \		
		\sum_{\mu \in \Sp(H)\setminus\{k\}} e^{-\ii t\mu}
		\left(\ebra{(v,w)}\eigenA{\mu}{G \star H}\eket{(v',w')}\right).
\end{aligned}
\end{equation}
Thus
\begin{equation} \label{eq::31}
\begin{aligned}
	\ebra{(v,0)} e^{-\ii tA(G\star H)} \eket{(v',0)}
	 &= \sum_{\lambda \in \Sp(G)\setminus\{0\}}
		e^{-\ii t(\lambda + k)/2}
		\ebra{v} \eigenA{\lambda}{G} \eket{v'}
			\sum_\pm e^{\mp \ii t\Lambda_{\lambda}/2}
				\frac{(\lambda^\pm - k)^2}{(\lambda^\pm - k)^2 + m \lambda^2}\\
     & \ + \
     \sum_{\lambda =0}e^{-\ii tk/2}\ebra{v} \eigenA{0}{G} \eket{v'}e^{\ii tk/2}.
\end{aligned}
\end{equation}
It is easy to verify that
\begin{equation}\label{eq::32}
\prod_{\pm} \bigl((\lambda^\pm - k)^{2} + m\lambda^2\bigr) = m \lambda^2 \Lambda_{\lambda}^{2} \mbox{ and }
\prod_{\pm} (\lambda^\pm - k) = -m \lambda^2,
\end{equation}
which leads to
\begin{equation}\label{eq::34}
\sum_{\pm} e^{\mp \ii t\Lambda_{\lambda}/2}
	\frac{(\lambda^\pm-k)^{2}}{(\lambda^\pm-k)^{2}+m \lambda^2}
	\ = \
	\cos(t\Lambda_{\lambda}/2) - \frac{\lambda - k}{\Lambda_{\lambda}} \ii\sin(t\Lambda_{\lambda}/2).
\end{equation}
By substituting \eqref{eq::34} into \eqref{eq::31}, Merging $\lambda =0$ and $\lambda \in \Sp(G)\setminus\{0\}$ into the result, we obtain the required result.

For (\ref{item::prop::3.3::2}) By Proposition \ref{prop::3.2} and \eqref{eq::30} we obtain

\begin{equation}\label{eq::28}
	\ebra{(v',0)} e^{-\ii tA(G \star H)} \eket{(v,w)}
	= \sum_{\lambda \in \Sp(G)} e^{-\ii t(\lambda + k)/2}
		\ebra{v}\eigenA{\lambda}{G}\eket{v'}
        \sum_\pm e^{\mp \ii t\Lambda_{\lambda}/2}
		 \frac{\lambda(\lambda^\pm - k)}{(\lambda^\pm - k)^2 + m \lambda^2}.
\end{equation}
By \eqref{eq::32}, we obtain
\begin{equation}\label{eq::29}
\sum_\pm e^{\mp \ii t\Lambda_{\lambda}/2}
		 \frac{\lambda(\lambda^\pm - k)}{(\lambda^\pm - k)^2 + m \lambda^2}=\frac{-2\lambda}{\Lambda_\lambda} \ii \sin\left(t\Lambda_\lambda/2 \right).
\end{equation}
By substituting \eqref{eq::29} into \eqref{eq::28}, we obtain the required result.
\end{proof}


\section{Perfect State Transfer}\label{sec::4}

According to Lemma \ref{lem::2.2}, we know that the periodicity is a necessary condition
for the existence of perfect state transfer, that is, if a graph is not periodic at any vertex, then it admits no perfect state transfer. In the section, we study the periodicity of neighborhood coronas, and show that there is no perfect state transfer in many kinds of neighborhood coronas.
\begin{lem} \label{lem::4.1}
Let $G$ be a graph and $H$ be a regular graph.
If $(v,w)$ is periodic in $G \star H$, then $(v,0)$ is periodic in $G \star H$.
\begin{proof}
By Proposition \ref{prop::3.2}, we see that the eigenvalue support of $(v,0)$ is contained in the eigenvalue support of $(v,w)$. Thus the results follows.
\end{proof}
\end{lem}

The following lemma gives an algebraic characterization for the periodicity of $(v,0)$ in  $G \star H$.
\begin{lem} \label{lem::4.2}
Let $G$ be a connected graph on $n$ vertices with $n \geqslant 2$ and $H$ be a $k$-regular graph on $m$ vertices.
Then $G \star H$ is  periodic at  $(v,0)$
if and only if the following conditions are satisfied:
\begin{enumerate}[(i)]
    \item \label{item::1} $k=0$;

    \item \label{item::2} $m=\frac{1}{4}(t-1)(t+1)$, where  $t > 1$ is an odd positive integer;

    \item \label{item::3} Each $\lambda \in \Phi_{v}$ is an integer or there exists a positive square-free integer $\Delta$
such that each $\lambda \in \Phi_{v}$ is a non-zero integer multiple of $\sqrt{\Delta}$.
\end{enumerate}
\end{lem}
\begin{proof}
By Proposition \ref{prop::3.2}, the eigenvalue support of $(v,0)$ in $G \star H$ is
\begin{equation}  \label{eq::35}
\Phi_{(v,0)} = \left\{\lambda^\pm=\dfrac{1}{2}\left(\lambda + k \pm \sqrt{(\lambda - k)^2 + 4m\lambda^2}\right):\lambda \in \Phi_{v}\right\}.
\end{equation}

\vspace{2mm}
\noindent{\bf{For the sufficiency part.}}  Suppose that the conditions (\ref{item::1}), (\ref{item::2}) and (\ref{item::3}) hold. Let $\lambda = a\sqrt{\Delta}$ with $a \in \Z \verb|\| \{0\}$, and $\Delta$ equal to 1 or positive square-free integer . By \eqref{eq::35}, we obtain $\lambda^\pm = a\sqrt{\Delta}(1 \pm \sqrt{1+4m})/2$.
Since $a(1 \pm \sqrt{1+4m})/2=a(1 \pm t)/2 \in \Z$, by Theorem \ref{lem::2.3},  $G \star H$  is periodic at $(v,0)$.

\vspace{2mm}		
\noindent{\bf{For the necessity part.}}
Suppose that $G \star H$ is periodic at $(v,0)$.
By Theorem \ref{lem::2.3}, it suffices to consider the following two situations.

\begin{case}
	For each $\lambda \in \Phi_{v}$, 		
	suppose that all elements of $\Phi_{(v,0)}$ are integers. Then $(\lambda^+ - \lambda^-)^{2} = (1+4 m)\lambda^2 - 2\lambda k + k^2$ are perfect square integers. Suppose $k\neq 0$, then $m\neq 0$ and $(1+4 m)\lambda^2 - 2\lambda k + k^2$ must be a perfect square. So $1+4 m =1$, then $m=0$, a contradiction. If $k=0$, then $(\lambda^+ - \lambda^-)^{2} = (1+4 m)\lambda^2$ are perfect square integers. So $\lambda$ is an integer and $1+4 m =t^2$, then $t = \sqrt{1+4m}> 1$ is an odd positive integer.
\end{case}
\begin{case}
	Suppose that each element of $\Phi_{(v,0)}$ is of the form
	\begin{equation*}
		\lambda^\pm = \frac{1}{2}(a + b_{\pm}\sqrt{\Delta}),
	\end{equation*}
    where $a$, $b_{\pm}$ are integers, and $\Delta$ is a square-free integer.
	Since $\lambda^+ + \lambda^- = \lambda + k = a + \frac{1}{2}(b_{+} + b_{-})\sqrt{\Delta}$, we have $\lambda = a - k + \frac{1}{2}(b_{+} + b_{-})\sqrt{\Delta}$.
According to the proof of Proposition \ref{prop::3.3},  $(\lambda^+-k)(\lambda^- -k)= -m \lambda^2$. Thus
	\begin{equation}\label{eq::36}
		\begin{aligned}
			\frac{1}{4}&\left((a-2k)^2 + b_{+} b_{-} \Delta
			+ (a-2k)(b_{+} + b_{-})\sqrt{\Delta}\right)\\
			&= -m \left((a-k)^2 + \frac{1}{4}(b_{+} + b_{-})^2 \Delta + (a-k)(b_{+} + b_{-})\sqrt{\Delta}\right).
		\end{aligned}
	\end{equation}
	Since $\sqrt{\Delta}$ is not an integer, by comparing the coefficients of Eq. \eqref{eq::36}, we get $\frac{1}{4} (a-2k) (b_{+} + b_{-}) = -m (a-k) (b_{+} + b_{-}). $
	If $b_{+} + b_{-}=0$, then $a = \lambda^+ + \lambda^- = \lambda + k$, and so  $|\Phi_{v}| = 1$, which is impossible because  $G$ is a connected graph on $n$ vertices with $n \geqslant 2$.
Thus $b_{+} + b_{-}\neq 0$, and  $a = k(4m+2) (4m+1)^{-1}$ is an integer.
	Since $k\leqslant m-1$, we have $a = k = 0$, which gives that $\lambda = \frac{1}{2} (b_{+} + b_{-}) \sqrt{\Delta}$ , and
	$\lambda^\pm = \frac{1}{2}b_{\pm}\sqrt{\Delta} = \frac{1}{2} \left( 1 \pm \sqrt{1 + 4m} \right)\lambda$. Thus $t = \sqrt{1 + 4m} > 1$ is an odd positive integer.
\end{case}
\end{proof}

For $m=\frac{1}{4}(t-1)(t+1)$ with odd positive integer $t > 1$, according to Lemma \ref{lem::4.2}, we find that the eigenvalues of $G \star \overline{K}_{m}$ are integer multiples of the eigenvalues of $G$. If $(v,0)$ is periodic in the neighborhood corona $G \star \overline{K}_{m}$, then $v$ is periodic in $G$. Therefore, we get the following lemma.
\begin{lem}\label{lem::4.3}
Let $G$ be a connected graph on $n$ vertices with $n \geqslant 2$. If $(v,0)$ is periodic in the neighborhood corona $G \star \overline{K}_{m}$, then $v$ is periodic in $G$, in which $m=\frac{1}{4}(t-1)(t+1)$ for odd positive integer $t > 1$.
\end{lem}
%
%
%
By Lemmas \ref{lem::2.2}, \ref{lem::4.1}, \ref{lem::4.2} and \ref{lem::4.3}, we can rule out some perfect state transfer in $G \star H$. Therefore, we can obtain the following some results.
\begin{thm}
Let $G$ be a connected graph on $n$ vertices with $n \geqslant 2$ and $H$ be a $k$-regular graph on $m$ vertices.
If $k=0$, $\sqrt{1+4m}$ is no integer or $k \geqslant 1$, then  $G \star H$ has no periodic vertices,
and therefore, has no perfect state transfer.
\end{thm}
\begin{proof}
For $k=0$, $\sqrt{1+4m}$ is no integer or $k \geqslant 1$, we can see that one of the conditions (\ref{item::1}),(\ref{item::2}) of the Lemma \ref{lem::4.2} is not satisfied. Therefore, $G \star H$ is no periodic at $(v,0)$. By Lemma \ref{lem::4.1}, $G \star H$ is no periodic at $(v,w)$, therefore, $G \star H$ has no periodic vertices. Hence, by Lemma \ref{lem::2.2}, $G \star H$ has no perfect state transfer.
\end{proof}

\begin{thm}\label{thm::4.6}
Let $G$ be a connected graph on $n$ vertices with $n \geqslant 2$ and $H$ be a $k$-regular graph on $m$ vertices.
If $G$ has no periodic vertices, then  $G \star H$ has no periodic vertices, and therefore, has no perfect state transfer.
\end{thm}
\begin{proof}
If $G$ has no periodic vertices, by Lemma \ref{lem::4.1} and \ref{lem::4.3}, then $G \star H$ has no periodic vertices. Hence, by Lemma \ref{lem::2.2}, $G \star H$ has no perfect state transfer.
\end{proof}

In what follows, we apply Proposition \ref{prop::3.3} to prove there is no perfect state transfer in $G \star H$.

\begin{thm}\label{thm::4.7}
Let $G$ be a connected graph on $n$ vertices with $n \geqslant 2$ and $H$ be a connected $k$-regular graph on $m$ vertices. If $v,v'$ are two distinct vertices of $G$, then there exists no perfect state transfer between $(v,0)$ and $(v',0)$ in $G \star H$; if $v,v'$ are two vertices (can be the same vertex) of $G$, and $w$ is a vertex of $H$, then there exists no perfect state transfer and between $(v',0)$ and $(v,w)$ in $G \star H$ .
\end{thm}

\begin{proof}
For vertices $v$ and $v'$ of $G$, by Proposition \ref{prop::3.3}(\ref{item::prop::3.3::1}), we have
\begin{equation*}
	\ebra{(v,0)} e^{-\ii tA(G \star H)} \eket{(v',0)}
	= \sum_{\lambda \in \Sp(G)} e^{-\ii t(\lambda + k)/2}
		\ebra{v}\eigenA{\lambda}{G}\eket{v'}
		\left( \cos\left(t\Lambda_\lambda/2\right)
			- \frac{\lambda - k}{\Lambda_\lambda} \ii \sin\left(t\Lambda_\lambda/2 \right) \right),
\end{equation*}
where $\Lambda_\lambda = \sqrt{(\lambda - k)^2 + 4m\lambda^2}$, and $\Sp(G)$ is the set of all distinct eigenvalues of $G$.

We get chain of inequalities
\begin{equation*}
\begin{split}
	\left|\ebra{(v,0)} e^{-\ii tA(G \star H)} \eket{(v',0)}\right|
	&\leq \sum_{\lambda \in \Sp(G)} \left|\ebra{v}\eigenA{\lambda}{G}\eket{v'}\right|\left|\cos\left(t\Lambda_\lambda/2\right)
			- \frac{\lambda - k}{\Lambda_\lambda} \ii \sin\left(t\Lambda_\lambda/2 \right)\right|\\
    &<\sum_{\lambda \in \Sp(G)} \left|\ebra{v}\eigenA{\lambda}{G}\eket{v'}\right|\leq 1.
\end{split}
\end{equation*}

  For vertices $v$ and $v'$ of $G$, and $w$ of $H$, by Proposition \ref{prop::3.3}(\ref{item::prop::3.3::2}), we have
\begin{equation*}
	\ebra{(v',0)} e^{-\ii tA(G \star H)} \eket{(v,w)}
	= \sum_{\lambda \in \Sp(G)} e^{-\ii t(\lambda + k)/2}
		\ebra{v}\eigenA{\lambda}{G}\eket{v'}
		 \frac{-2\lambda}{\Lambda_\lambda} \ii \sin\left(t\Lambda_\lambda/2 \right),
\end{equation*}
where $\Lambda_\lambda = \sqrt{(\lambda - k)^2 + 4m\lambda^2}$, and $\Sp(G)$ is the set of all distinct eigenvalues of $G$.

We get chain of inequalities
\begin{equation*}
\begin{split}
	\left|\ebra{(v',0)} e^{-\ii tA(G \star H)} \eket{(v,w)}\right|
	&\leq \sum_{\lambda \in \Sp(G)}
		\left|\ebra{v}\eigenA{\lambda}{G}\eket{v'}\right|
		 \left|\frac{-2\lambda}{\Lambda_\lambda} \ii \sin\left(t\Lambda_\lambda/2 \right)\right|\\
   &<\sum_{\lambda \in \Sp(G)} \left|\ebra{v}\eigenA{\lambda}{G}\eket{v'}\right|\leq 1.
\end{split}
\end{equation*}
\end{proof}
%

%
%
%
%

\section{Pretty Good State Transfer}\label{sec::5}

In this section, we consider pretty good state transfer for the neighborhood corona $G \star H$.
\begin{thm} \label{thm::5.1}
Let $G$ be a connected graph, $u,v$ be two vertices of $G$, and $H$ be a connected $k$-regular graph $(k \neq 0)$ on $m$ vertices.
Suppose that there exists perfect state transfer between $u$ and $v$ at time $t = \pi/g$, for some positive integer $g$,
and that $0$ is not in the eigenvalue support of $u$. Then there exists pretty good state transfer
between $(u,0)$ and $(v,0)$ in $G \star H$.

\begin{proof}
Let $\Phi_{u}$ be the eigenvalue support of $u$ in $G$.
By Theorem \ref{thm::2.1},
if $G$ has perfect state transfer at time $\pi/g$ between the vertices $u$ and $v$, for some integer $g$,
then all eigenvalues in $\Phi_{u}$ must be integers.
For each eigenvalue $\lambda$ in $\Phi_{u}$, let $c_\lambda$ be the square-free part of $(\lambda - k)^2 + 4 m \lambda^2$, 	
so that $\Lambda_\lambda = \sqrt{(\lambda - k)^2 + 4 m \lambda^2} = s_\lambda \sqrt{\mathstrut c_\lambda}$ for some integers $s_\lambda$.
Since $0 \not\in\Phi_{u}$,
then $\Lambda_\lambda$ is irrational and $c_\lambda > 1$ for each $\lambda$ in $S$.

By Corollary \ref{cor::2.6},
$\{\sqrt{\mathstrut c_\lambda} : \lambda \in \Phi_{u}\} \cup \{1\}$
is linearly independent over $\mathbb Q$.
By Theorem \ref{thm::2.4}, this implies that there exists integers $\ell, q_\lambda$ such that
\begin{equation*}
	\ell \sqrt{\mathstrut c_\lambda} - q_\lambda \ \approx \ -\frac{\sqrt{\mathstrut c_\lambda}}{2g}.
\end{equation*}
Multiplying by $4s_\lambda$ yields that
\begin{equation*}
	\left(4\ell + \frac{2}{g}\right) \Lambda_\lambda \ \approx \ 4 q_\lambda s_\lambda.
\end{equation*}
Therefore, if $t = (4\ell + 2/g)\pi$, then  $\cos(\Lambda_\lambda t /2) \approx 1$ for each $\lambda$ in $S$.
By Proposition \ref{prop::3.3}(\ref{item::prop::3.3::1}), we have
\begin{equation*}
\begin{split}
\left|\ebra{(u,0)} e^{-\ii tA(G\star H)} \eket{(v,0)}\right|
	& =\left| \sum_{\lambda \in \Sp(G)} e^{-\ii t(\lambda+k)/2} \left(\cos(\Lambda_\lambda t/2)
	- \frac{\lambda-k}{\Lambda_\lambda} \ii \sin(\Lambda_\lambda t/2) \right) \ebra{u} \eigenA{\lambda}{G} \eket{v}\right| \\
	& \approx\left| \sum_{\lambda \in \Sp(G)} e^{- \ii (2\pi) \ell \lambda} e^{- \ii t k/2}
          e^{-\ii (\pi/g) \lambda}
		\ebra{u}\eigenA{\lambda}{G} \eket{v} \right|\\
	& =\left| e^{- \ii t k/2}  \ebra{u} e^{-\ii (\pi/g) A(G)} \eket{v}\right|\\
    & =\left| e^{- \ii (\pi/g) k}  \ebra{u} e^{-\ii (\pi/g) A(G)} \eket{v}\right|\\
    & =1,
\end{split}
\end{equation*}
which implies that there exists pretty good state transfer
between $(u,0)$ and $(v,0)$ in $G \star H$.
\end{proof}
\end{thm}

\begin{thm} \label{thm::5.2}
Let $G$ be a connected graph having zero as an eigenvalue and $H$ be a connected $k$-regular graph $(k \neq 0)$ on $m$ vertices. Suppose that $G$ has perfect state transfer at time
$\pi/2$ between vertices $u$ and $v$. Then there is pretty good state transfer between
$(u,0)$ and $(v,0)$ in the neighborhood corona $G \star H$.

\begin{proof}
Similar to the proof of Theorem \ref{thm::5.1}, for each $\lambda \in \Phi_{u}$, we can write $\Lambda_\lambda = s_\lambda \sqrt{\mathstrut c_\lambda}$ where
$c_\lambda$ is the square-free part of $(\lambda - k)^2 + 4 m \lambda^2$ and $s_\lambda$ is an integer.
Note that $c_\lambda = 1$ if and only if $\lambda=0$.
The set $\{ \sqrt{\mathstrut c_\lambda} : \lambda \in \Phi_{u}, \lambda\neq 0\} \cup \{1\}$
is linearly independent over $\mathbb Q$.

For $\lambda \neq 0$, by Theorem \ref{thm::2.4}, this implies that there exist integers $l$ and $q_\lambda$ such that
\begin{equation*}
	\ell \sqrt{\mathstrut c_\lambda} - q_\lambda \ \approx \ -\frac{\sqrt{\mathstrut c_\lambda}}{4}+\frac{1}{2s_\lambda}.
\end{equation*}
If $t=(4\ell +1)\pi$, then  $\cos(\Lambda_0 t/2)=-1$,
and $\cos(\Lambda_\lambda t/2) \approx -1$ for $\lambda \neq 0$.
By Proposition \ref{prop::3.3}(\ref{item::prop::3.3::1}), we have
\begin{equation*}
\begin{split}
\left|\ebra{(u,0)} e^{-\ii tA(G\star H)} \eket{(v,0)}\right|
	& = \left|\sum_{\lambda \in \Sp(G)} e^{-\ii t(\lambda+k)/2} \left(\cos(\Lambda_\lambda t/2)
	- \frac{\lambda-k}{\Lambda_\lambda} \ii \sin(\Lambda_\lambda t/2) \right) \ebra{u} \eigenA{\lambda}{G} \eket{v}\right| \\
	& \approx\left| - \sum_{\lambda \in \Sp(G)} e^{- \ii (2\pi) \ell \lambda} e^{- \ii t k/2}
e^{-\ii(\pi/2)\lambda}
		\ebra{u}\eigenA{\lambda}{G} \eket{v} \right|\\
	& =\left| - e^{- \ii t k/2} \ebra{u} e^{-\ii (\pi/2) A(G)} \eket{v}\right|\\
    & =\left| - e^{- \ii (\pi/2) k}  \ebra{u} e^{-\ii (\pi/2) A(G)} \eket{v}\right|\\
    & =1,
\end{split}
\end{equation*}
which implies that there exists pretty good state transfer between $(u,0)$ and $(v,0)$ in $G \star H$.
\end{proof}
\end{thm}

%
%
%
%
%
From the above conclusion, we find that if $G$ has perfect state transfer, then $G \star H$ has pretty good state transfer. The following suggests that if there exists no perfect state transfer in $G$ can obtain the existence of  pretty good state transfer in $G \star H$.


\begin{lem}[\cite{cggv15}, Lemma 4.4]  \label{lem::5.4}
Let $G$ be a distance-regular graph of diameter $d$ with eigenvalues $\lambda_0 > \lambda_1 > \cdots > \lambda_d$.
Let the spectral decomposition of the adjacency matrix of $G$ be given by
$A(G) = \sum_{j=0}^d \lambda_j \eigenA{j}{G}$.
Suppose that $G$ is antipodal with classes of size two. 
Then
\begin{equation*}
	A_d(G) \eigenA{j}{G} = (-1)^j \eigenA{j}{G}.
\end{equation*}
Here, $A_d(G)$ is the adjacency matrix of a graph obtained from $G$ by connecting vertices $u$ and $v$ if and only if they are at distance $d$.
\end{lem}

\begin{thm}
Let H be a $k$-regular connected graph $(k \neq 0)$ on $m$ vertices, $n \geqslant 3$ be an odd integer, and $u$ and $v$ be antipodal vertices of the cocktail party graph $\overline{nK_2}$.
Then there is pretty good state transfer between $(u,0)$ and $(v,0)$ in $\overline{nK_2} \star H$.
\end{thm}
\begin{proof}
The eigenvalues of $\overline{nK_2}$ are
$\lambda_0 = 2n-2$, $\lambda_1 = 0$, and $\lambda_2 = -2$. If $u$ and $v$ are antipodal vertices of $\overline{nK_2}$,
by Lemma \ref{lem::5.4}, then we have $\eigenA{j}{\overline{nK_2}}=(-1)^{j}A_2 \eigenA{j}{\overline{nK_2}}$. Because $A_2$ is a permutation
matrix, then we have
\begin{equation} \label{eq::39}
	\ebra{u} \eigenA{j}{\overline{nK_2}} \eket{v} = (-1)^j \ebra{u} \eigenA{j}{\overline{nK_2}} \eket{u}
\end{equation}
for $j = 0,1,2$. By Proposition \ref{prop::3.3}(\ref{item::prop::3.3::1}),
$\Lambda_{j} = \sqrt{(\lambda_{j}-k)^{2} + 4 m \lambda_{j}^{2}}$. Now we pick  $t$ such that it suffices to approximate
\begin{equation} \label{eq::40}
e^{-\ii t(\lambda_j+k)/2} \approx 1,
\end{equation}
and
\begin{equation} \label{eq::41}
	\cos\left(\Lambda_j t/2\right) \approx (-1)^{j+1}.
\end{equation}

Let $t=8\ell\pi$ with $\ell \in \mathbb{Z}$. Then $e^{-\ii t(\lambda_j+k)/2} = 1$ and $\cos(\Lambda_1 t /2) = 1$.  Also note that
$\Lambda_0 = \sqrt{(2n-2-k)^2+16m(n-1)^2}$ and
$\Lambda_2 = \sqrt{(k+2)^2+16m}$. We apply Theorem \ref{thm::2.4} to $\Lambda_0$ and $\Lambda_2$. Let $c_0$ and $c_2$ denote the square-free part of $\Lambda_0^2$ and $\Lambda_2^2$. Then $\Lambda_0 = s_0\sqrt{\mathstrut c_0}$ and  $\Lambda_2 = s_2\sqrt{\mathstrut c_2}$ for some odd integer $s_0$ and $s_2$. We have
\begin{align*}
	\ell \sqrt{\mathstrut c_0} - q_0 \approx \frac{1}{4}, \\
	\ell \sqrt{\mathstrut c_2} - q_2 \approx \frac{1}{4}.
\end{align*}
At time $t = 8\ell\pi $, we have $t \Lambda_0/2 \approx 4\pi q_0 s_0 + \pi s_0$ and $t\Lambda_2/2 \approx 4\pi q_2 s_2 + \pi s_2$,
which  leads to \eqref{eq::41}.
Form \eqref{eq::39}, \eqref{eq::40}, \eqref{eq::41} and Proposition \ref{prop::3.3}(\ref{item::prop::3.3::1}), we have
\begin{equation*}
	\begin{split}
		\left|\ebra{(u,0)} e^{-\ii tA(\overline{nK_2}\star H)} \eket{(v,0)}\right|
		& = \left|\sum_{j=0}^{2} e^{-\ii t(\lambda_j+k)/2} \left(\cos(\Lambda_j t/2)
		- \frac{\lambda_j-k}{\Lambda_j} \ii \sin(\Lambda_j t/2) \right) \ebra{u} \eigenA{j}{\overline{nK_2}} \eket{v}\right| \\
		& \approx\left|\sum_{j=0}^{2} (-1)^{j+1} \ebra{u} \eigenA{j}{\overline{nK_2}} \eket{v} \right|\\
		& =\left|\sum_{j=0}^{2} (-1)^{j+1} (-1)^{j} \ebra{u} \eigenA{j}{\overline{nK_2}} \eket{u} \right|\\
        & =\left|-\ebra{u}\left(\sum_{j=0}^{2} \eigenA{j}{\overline{nK_2}}\right) \eket{u} \right|\\
		& =1,
	\end{split}
\end{equation*}
which implies that there exists pretty good state transfer between $(u,0)$ and $(v,0)$ in $\overline{nK_2} \star H$.

\end{proof}


\end{document}